\documentclass[11pt]{article} 
\usepackage{amssymb, latexsym, amsmath}
\newtheorem{defin}{}
\newtheorem{saetze}[defin]{}
\newtheorem{conjec}[defin]{}
\newtheorem{lemmas}[defin]{}
\newtheorem{folger}[defin]{}
\newtheorem{bemerk}[defin]{}

\newenvironment{theorem}  {\begin{saetze}\it {\bf Theorem:}}{\end{saetze}}

\newenvironment{lemma}    {\begin{lemmas}\it {\bf Lemma:}}{\end{lemmas}}
\newenvironment{corollary}{\begin{folger}\it {\bf Corollary:}}{\end{folger}}
\newenvironment{remark}   {\begin{bemerk}\rm {\it Remark:}}{\end{bemerk}}
\newenvironment{proof}    {\noindent{\it Proof}:}{{\hfill \fillbox \bigskip}}

\newcommand{\fillbox}{\mbox{$\bullet$}}
\newcommand{\ra}{\rightarrow}

\newcommand{\ms}{\mapsto}
\newcommand{\ol}{\overline}
\newcommand{\ti}{\tilde}
\newcommand{\N}{\mathbb N}

\newcommand{\Z}{\mathbb Z}
\newcommand{\Q}{\mathbb Q}
\newcommand{\T}{\mathcal{T}}

\newcommand{\Aut}{{\mathrm{Aut}}}
\newcommand{\GL}{{\mathrm{GL}}}

\newcommand{\g}{\mathfrak{g}}

\newenvironment{items}{\begin{list}{$\alph{item})$}
{\labelwidth18pt \leftmargin18pt \topsep3pt \itemsep1pt \parsep0pt}}
{\end{list}}

\begin{document}

\title{Torsion-free nilpotent groups of small Hirsch length \\
       with isomorphic finite quotients}
\author{Alexander Cant and Bettina Eick}
\date{May 5, 2023}

\maketitle

\begin{abstract}
Let $\T$ denote the class of finitely generated torsion-free nilpotent
groups. For a group $G$ let $F(G)$ be the set of isomorphism classes of
finite quotients of $G$. 
Pickel proved that if $G \in \T$, then the set
$\g(G)$ of isomorphism classes of groups $H \in \T$ with $F(G) = F(H)$ is
finite. We give an explicit description of the sets $\g(G)$ for the 
$\T$-groups $G$ of Hirsch length at most $5$. Based on this, we show that 
for each Hirsch length $n \geq 4$ and for each $m \in \N$ there is a 
$\T$-group $G$ of Hirsch length $n$ with $|\g(G)| \geq m$.
\end{abstract}

\section{Introduction}

Let $G$ be a finitely generated torsion-free nilpotent group, $\T$-group
for short, and write $F(G)$ for the set of isomorphism classes of finite
quotients of $G$. The {\em genus} $\g(G)$ of $G$ (in the class of all 
$\T$-groups) is the set of isomorphism classes of $\T$-groups $H$ with 
$F(G) = F(H)$. A purpose of this paper is to determine explicitly the 
genus $\g(G)$ for each $\T$-group $G$ of Hirsch length at most $5$.
This addresses the ``Basic Problems'' as introduced by Grunewald \&
Zalesskii \cite{GrZa11} for $\T$-groups.

Pickel \cite{Pic71, Pic71a} proved the interesting theorem that if $G$ is 
a $\T$-group, then its genus $\g(G)$ is finite. He obtains this result by 
showing that the elements in $\g(G)$ correspond to a subset of the double 
cosets in a subgroup of $\prod_{p \text{ prime}} \Aut(\Q_p G)$, where 
$\Q_p G$ is the rational $p$-adic completion of $G$. While this allows to 
prove the desired finiteness, it does not yield many insights into the 
nature of the groups in $\g(G)$.

Our results here yield an explicit description of the genera of the 
$\T$-groups of Hirsch length at most $5$. This provides an illustration
of Pickel's theorem \cite{Pic71, Pic71a} and gives some
insight into the nature of the groups in $\g(G)$ for a $\T$-group
$G$ of Hirsch length at most $5$. Our methods rely on the classification 
up to isomorphisms of the $\T$-groups of Hirsch length at most $5$ by 
Eick \& Engel \cite{EEn17}.

Grunewald \& Scharlau \cite{GSc79} proved that the genus of a $\T$-group 
of Hirsch length at most $5$ and class at most $2$ is trivial; that is, 
it consists of a single isomorphism class only. Thus the genus of a 
$\T$-group of Hirsch length at most $3$ is trivial.
Our results here imply the following, see Corollary \ref{coro} and Section 
\ref{nobound} below.

\begin{theorem}
\label{unbounded}
For each $n \geq 4$ and each $m \in \N$ there exists a $\T$-group $G$
of Hirsch length $n$ whose genus $\g(G)$ has more than $m$ elements.
\end{theorem}

Grunewald \& Scharlau \cite{GSc79} proved a that for each $m \in \N$ there
exists a $\T$-group $G$ of Hirsch length $6$ whose genus $\g(G)$ has more 
than $m$ elements. Theorem \ref{unbounded} extends this result
to all possible Hirsch lengths. We also note that the groups used in the 
proof of Theorem \ref{unbounded} have class $3$, while the groups in 
\cite{GSc79} have class $2$. 

Based on work by Pickel \cite{Pic73}, Grunewald, Pickel \& Segal 
\cite{GrPiSe80} proved that the genus in the class of polycyclic-by-finite 
groups is finite. In contrast, the work by Nekrashevych \cite{Nek14} as 
well as by Nikolov \& Segal \cite{NSe21} shows that there are classes of 
groups that allow uncountably large genera.

\section{Results}\label{results}

Let $G$ be a $\T$-group of Hirsch length $n$ and class $c$. We denote with 
$G = \gamma_1(G) > \ldots > \gamma_c(G) > \gamma_{c+1}(G) = \{1\}$ the lower
central series of $G$ and for $1 \leq i \leq c+1$ we write 
$I_i(G)/\gamma_i(G)$ for the torsion subgroup of $G/\gamma_i(G)$. Then 
\[ G = I_1(G) > I_2(G) > \ldots > I_c(G) > I_{c+1}(G) = \{1\} \]
is the {\em isolator series} of $G$. It is a central series with torsion-free
abelian quotients. The sequence $(d_1, \ldots, d_c)$ of ranks of 
these quotients is the {\em type} of $G$.

We refine the isolator series of $G$ to a central series
$G = G_1 > G_2 > \ldots > G_n > G_{n+1} = \{1\}$ whose quotients are infinite cyclic.
For all $i$, let $g_i G_{i+1}$ be a generator of $G_i/G_{i+1}$.
Then $(g_1, \ldots, g_n)$ is known as a {\em basis} of $G$.
Note that every group element has a unique representation of the form
$g_1^{x_1} \cdots g_n^{x_n}$ with integer exponents $x_1,\ldots,x_n$.
Throughout the article, let $\Z(n)$ denote the set of all sequences
$(t_{i,j,k} \in \Z \mid 1 \leq i < j < k \leq n)$.
Each basis of $G$ induces a finite presentation for $G$ of the form
\[ G(t) = \langle \, g_1, \ldots, g_n \mid 
  [g_j, g_i] = g_{j+1}^{t_{i,j,j+1}} \cdots g_n^{t_{i,j,n}}
  \mbox{ for } 1 \leq i < j \leq n \, \rangle \]
for some $t \in \Z(n)$.
As $n \leq 5$ in all our applications, we write $t_{ijk}$ for $t_{i,j,k}$.

\subsection{Type $(2,1,1)$}\label{sec:Type211}

For $t \in \Z(4)$ define $d(t) = \gcd(t_{123}, t_{134})$. Further, we write 
\[ T(2,1,1) = \{ \, t \in \Z(4) \mid 0 < t_{123}, t_{134}; \
    0 \le t_{124} \leq \tfrac{d(t)}{2}; \ t_{234} = 0 \, \}.\]
Part (a) of the following theorem follows from the work of Eick \& Engel 
\cite{EEn17} and part (b) is our first main result here.

\begin{theorem}\label{class211}
\begin{items}
\item[\rm (a)]
For each $\T$-group $G$ of type $(2,1,1)$ there exists a unique $t \in 
T(2,1,1)$ with $G \cong G(t)$.
\item[\rm (b)]
Let $s, t \in T(2,1,1)$. Then $F(G(s)) = F(G(t))$ if and only if
\begin{items}
\item[$\bullet$] 
$s_{123} = t_{123}$ and $s_{134} = t_{134}$, and 
\item[$\bullet$]
for each maximal prime power $p^k$ dividing $d(s)$ there exists 
$w_p \in \Z$ with $p \nmid w_p$ and $p^k \mid (t_{124} w_p - s_{124})$.
\end{items}
\end{items}
\end{theorem}

For a prime $p$ let $t_i \in T(2,1,1)$ be defined by $t_{i,123} = t_{i,134} 
= p$ and $t_{i,124} = i$. Then $\{ G(t_i) \mid 1 \leq i \leq \lfloor 
\frac{p}{2} \rfloor \}$ consists of pairwise non-isomorphic groups with 
the same finite quotients and hence $|\g(G(t_i))| \geq \lfloor \frac{p}{2} 
\rfloor$.  In Section \ref{proofs} we observe that if $G$ is of type 
$(2,1,1)$ and $A$ is free abelian of finite rank, then $|\g(G \times A)| 
\geq |\g(G)|$ follows. Hence we obtain the following which, in turn, proves 
Theorem \ref{unbounded}.

\begin{corollary}
\label{coro}
For each $n \geq 2$, there is no finite upper bound to the size of the 
genus of a $\T$-group of type $(n,1,1)$.
\end{corollary}

\subsection{Type $(3,1,1)$}

For $t \in \Z(5)$ define $d_1(t) = \gcd(t_{145}, t_{235})$ and
$d_2(t) = \gcd(t_{124}, t_{135}, d_1(t))$. Further, we write
\begin{align*}
    T(3,1,1) = \{ \, t \in \Z(5) \mid
        & \; 0 < t_{124}, t_{145}; \ 0 \le t_{235}; 
          \ 0 \leq t_{135} \leq \tfrac{d_1(t)}{2}; 
          \ 0 \leq t_{125} \leq \tfrac{d_2(t)}{2}; \\
        & \; t_{ijk} = 0 \mbox{ otherwise} \, \}.
\end{align*}
Part (a) of the following theorem follows from the work of Eick \& Engel 
\cite{EEn17} and part (b) is our second main result here.

\begin{theorem}\label{class311}
\begin{items}
\item[\rm (a)]
For each $\T$-group $G$ of type $(3,1,1)$ there exists a unique $t \in 
T(3,1,1)$ with $G \cong G(t)$.
\item[\rm (b)]
Let $s, t \in T(3,1,1)$. Then $F(G(s)) = F(G(t))$ if and only if
\begin{items}
\item[$\bullet$] 
$s_{ijk} = t_{ijk}$ for $(ijk) \in \{(124), (145), (235) \}$,
\item[$\bullet$]
for each maximal prime power $p^k$ dividing $d_1(s)$ there exists $w_p \in 
\Z$ with $p \nmid w_p$ and $p^k \mid (t_{135} w_p - s_{135})$, and
\item[$\bullet$]
for each maximal prime power $p^k$ dividing $d_2(s)$ there exists $v_p \in 
\Z$ with $p \nmid v_p$ and $p^k \mid (t_{125} v_p - s_{125})$.
\end{items}
\end{items}
\end{theorem}

\subsection{Type $(2,1,1,1)$}

For $t \in \Z(5)$ define $d_3(t) = \gcd(t_{123}, t_{145}, t_{235})$.
Consider the set
\begin{align*}
    T_0(2,1,1,1) = \{ \, t \in \Z(5) \mid
        & \; 0 < t_{123}, t_{134}, t_{145}; 
          \ 0 \le t_{124}, t_{125}, t_{135}, t_{235}; \\
        & \; t_{ijk} = 0 \mbox{ otherwise} \, \}.
\end{align*}
We define a relation on $T_0(2,1,1,1)$ over a ring $R$ that is either 
$\Z$ or $\Z_p$ via $s \sim_R t$ if $s_{ijk} = t_{ijk}$ for $(ijk) \in 
\{(123), (134), (145), (235)\}$ and there exist $u \in R^*$ and 
$v,w,x,y,z \in R$ so that
\begin{align*}
    t_{124} u &= s_{124} - s_{123} w + s_{134} v, \\
    t_{135} u &= s_{135} + s_{145} w + s_{134} x + s_{235} y, 
                 \mbox{ and } \\
    t_{125} u^2 &= s_{125} + s_{135} v + s_{124} x + s_{134} vx
                 + d_3(s) z.
\end{align*}
We note that this is an equivalence relation that corresponds to an 
action of the group $R^5 \rtimes R^*$ on $T_0(2,1,1,1)$. Let $T(2,1,1,1)$ 
be a set of representatives for $\sim_\Z$. Part (a) of the following 
theorem follows from the work of Eick \& Engel \cite{EEn17} and part (b) 
is our third main result here.

\begin{theorem}\label{class2111}
\begin{items}
\item[\rm (a)]
For each $\T$-group $G$ of type $(2,1,1,1)$ there exists a unique $t \in 
T(2,1,1,1)$ with $G \cong G(t)$.
\item[\rm (b)]
Let $s, t \in T(2,1,1,1)$. Then $F(G(s)) = F(G(t))$ if and only if
$s \sim_{\Z_p} t$ for all primes $p$, or, equivalently, if and only if
\begin{items}
\item[$\bullet$] 
$s_{ijk} = t_{ijk}$ for $(ijk) \in \{(123), (134), (145), (235)\}$, and
\item[$\bullet$]
for each prime $p$ there are $u,v,w,x,y,z \in \Z$ with $p \nmid u$ so that 
with $p^k$ the maximal $p$-power dividing the product $s_{134} d_3(s)$ it follows that
\begin{align*}
 t_{124} u & \equiv s_{124} - s_{123} w + s_{134} v & \bmod \ p^k, \\
 t_{135} u & \equiv s_{135} + s_{145} w + s_{134} x + s_{235} y & 
   \bmod \ p^k, & \quad \text{and} \\
 t_{125} u^2 & \equiv s_{125} + s_{135} v + s_{124} x + s_{134} vx
   + d_3(s) z & \bmod \ p^k.
\end{align*}
\end{items}
\end{items}
\end{theorem}

\subsection{Type $(2,1,2)$}

For $k \in \N$ and $R \in \{\Z, \Z_p\}$ let $D_k(R)$ denote the congruence 
subgroup
\[ D_k(R) = \left\{ \, \begin{pmatrix} a_{11} & a_{12} \\ 
                  a_{21} & a_{22} \end{pmatrix}
            \; \middle| \; k \mid a_{12} \, \right\} \leq \GL(2, R). \]
For $a,b,c \in \N$ with $a \mid b$ define
$L(a,b,c) = \{ \, (x,y) \in \N_0^2 \mid x < \gcd(a,c), \ y < \gcd(b,c) \, \}$.
Then $L(a,b,c)$ is a finite set and the group $\GL(2, \Z_p)$
acts on this set via 
\[ (x,y) \ms (x,y)M \mod (\gcd(a,c), \gcd(b,c)), \]
where the mod-reduction is carried out independently in both components.
Let $O(a,b,c)$ denote a set of orbit representatives for the action of
$D_{b/a}(\Z)$ on $L(a,b,c)$.
We write
\begin{align*}
    T(2,1,2) = \{ \, t \in \Z(5) \mid
        & \; 0 < t_{123}, t_{134}, t_{235}; \  t_{134} \mid t_{235}; \  
           (t_{124}, t_{125}) \in O(t_{134}, t_{235}, t_{123}); \\
        & \; t_{ijk} = 0 \text{ otherwise} \, \}.
\end{align*}
Part (a) of the following theorem follows from the work of Eick \& Engel 
\cite{EEn17} and part (b) is our fourth main result here.

\begin{theorem}\label{class212}
\begin{items}
\item[\rm (a)]
For each $\T$-group $G$ of type $(2,1,2)$ there exists a unique $t \in 
T(2,1,2)$ with $G \cong G(t)$.
\item[\rm (b)]
Let $s, t \in T(2,1,2)$ and $k = s_{235}/s_{134}$. Then 
$F(G(t)) = F(G(s))$ if and only if
\begin{items}
\item[$\bullet$] 
$t_{ijk} = s_{ijk}$ for $(ijk) \in \{ (123), (134), (235) \}$, and
\item[$\bullet$]
for every prime $p$ the tuples $(s_{124}, s_{125})$ and $(t_{124}, t_{125})$ 
are in the same orbit under the action of $D_k(\Z_p)$ on 
$L(t_{134}, t_{235}, t_{123})$.
\end{items}
\end{items}
\end{theorem}

An explicit description of the orbits in Theorem \ref{class212}(b)  
is given in Remark \ref{OrbitsOnL}.

\section{Methods}\label{methods}

This section describes the principal ideas used to prove the above results.
Let $G$ be a $\T$-group with basis $(g_1, \ldots, g_n)$. For 
$x = (x_1, \ldots, x_n) \in \Z^n$ we write $g^x$ as a short form of
$g_1^{x_1} \cdots g_n^{x_n}$. The multiplication,
inversion and powering in $G$ can be represented by functions
$f_j \colon \Z^n \times \Z^n \ra \Z$ and 
$k_j \colon \Z \times \Z^n \ra \Z$ defined via
\[ g_1^{f_1(x,y)} \cdots g_n^{f_n(x,y)} = g^x g^y \quad \mbox{and} \quad
 g_1^{k_1(\ell,x)} \cdots g_n^{k_n(\ell,x)} = (g^x)^\ell. \]
Hall \cite{Hal69} proved that the 
functions $f_1, \ldots, f_n$ and $k_1, \ldots, k_n$ can be described 
by rational polynomials, nowadays called Hall polynomials. We note the 
following.

\begin{lemma}\label{unique}
The Hall polynomials of a $\T$-group with respect to a fixed basis are unique.
\end{lemma}

\begin{proof}
This follows directly from the observation that a multivariate rational
polynomial $h \in \Q[z_1, \ldots, z_m]$ with
$h(z) = 0$ for all $z = (z_1, \ldots, z_m) \in \Z^m$ is the zero polynomial. 
\end{proof}

As a next step we extend the ring of exponents of group elements from $\Z$
to the rational numbers $\Q$, the $p$-adic integers $\Z_p$ or the $p$-adic
rationals $\Q_p$. The following lemma provides a starting point for this.

\begin{lemma}\label{contained}
Let $f \in \Q[z_1, \ldots, z_m]$ so that $f(a) \in \Z$ for all $a = (a_1,
\ldots, a_m) \in \Z^m$ and let $R = \Q$, $\Z_p$ or $\Q_p$. Then $f(a) \in R$ 
for all $a \in R^m$.
\end{lemma}

\begin{proof}
The cases $R = \Q$ and $R = \Q_p$ are obvious, since $f$ has rational 
coefficients. It remains to consider the case $R=\Z_p$. It is well-known
that a rational polynomial which takes integer values on integer inputs is
a $\Z$-linear combination of polynomials of the form
\[ f_{s_1, \ldots, s_m}(z_1, \ldots, z_m) :=
    \binom{z_1}{s_1} \cdots \binom{z_m}{s_m},\]
where $s_1, \ldots, s_m \in \N_0$, $z_1, \ldots, z_m$ are indeterminates and
\[ \binom{z_i}{s_i} = \frac{z_i (z_i-1) \cdots (z_i-s_i+1)}{s_i!}. \]
Clearly, $\binom{a}{s} \in \Q_p$ holds for $a \in \Z_p$ and $s \in \N_0$.
By definition, $a$ is the limit of a Cauchy sequence $(a_i)_{i \in \N}$ of 
non-negative integers with respect to the $p$-adic absolute value. Further 
$\binom{a_i}{s} \in \Z$ for all $i$. As polynomial functions are continuous and 
$\Z_p$ is closed and therefore complete, we obtain that $\binom{a}{s} \in \Z_p$.
\end{proof}

\subsection{Completions}

Let $G$ be a $\T$-group with basis $(g_1, \ldots, g_n)$. Then 
$G = \{ g^a \mid a \in \Z^n\}$. For $R = \Q$, $\Z_p$ or $\Q_p$ we define
the set $RG$ via
\[ RG = \{ g^a \mid a \in R^n \}.\]
The Hall polynomials $f_1, \ldots, f_n$ of $G$ facilitate a multiplication 
on $RG$: For $g^a, g^b \in RG$ we define $g^a \cdot g^b := g^{f(a,b)} \in RG$,
where $f(a,b) = (f_1(a,b), \ldots, f_n(a,b))$.

\begin{lemma}\label{lem:GroupRG}
Let $G$ be a $\T$-group and $R = \Q$, $\Z_p$ or $\Q_p$. 
\begin{items}
\item[\rm (a)]
$RG$ equipped with the multiplication via Hall polynomials is a group.
\item[\rm (b)]
The isomorphism type of the group $RG$ is independent of the chosen basis.
\end{items}
\end{lemma}

\begin{proof}
(a) The multiplication is well-defined by Lemmas \ref{unique} and 
\ref{contained}. It is not difficult to observe that an identity element
and inverses are contained in $RG$. The proof of Lemma \ref{unique} 
implies that the multiplication is associative, as the multiplication 
in $G$ is associative. \\
(b) Let $(g_1, \ldots, g_n)$ be a basis for $G$ and let 
$(h_1, \ldots, h_n)$ be a basis for another $\T$-group $H$.
Suppose that $\varphi \colon G \to H$ is an isomorphism.
We show that $\varphi$ extends to an isomorphism $RG \ra RH$.
Applying this with $G = H$ and two different bases then yields 
the desired result.

For $1 \leq i \leq n$ let $w_i = \varphi(g_i) = h^{r_i}$ with $r_i \in \Z^n$.
The Hall polynomials $\ol{f}_1, \ldots, \ol{f}_n$ of $H$ and the values
$r_1, \ldots, r_n$ can be used to determine rational polynomials $m_1, 
\ldots, m_n$ satisfying that
$\varphi(g^x) = w_1^{x_1} \cdots w_n^{x_n} = h_1^{m_1(x)} \cdots h_n^{m_n(x)}$
holds for all $x \in \Z^n$.
Since $\varphi$ is a group homomorphism, it follows that 
\[ h^{m(f(x,y))} = \varphi(g^{f(x,y)}) = \varphi(g^x g^y) 
= \varphi(g^x) \varphi(g^y) = h^{m(x)} h^{m(y)} = h^{\ol{f}(m(x),m(y))} \]
and thus $m(f(x,y)) = \ol{f}(m(x),m(y))$ for all $x,y \in \Z^n$.
This equation extends to $x,y \in R^n$ and thus $\varphi$ extends to a 
group homomorphism $\tilde{\varphi} \colon RG \ra RH$. 
Similarly, the inverse of $\varphi$ extends to a homomorphism $RH \to RG$,
and this extension inverts $\tilde{\varphi}$. Thus $\tilde{\varphi}$ is
an isomorphism.
\end{proof}

A group $H$ is called {\em $\Q$-powered} if for every $h \in H$ and every
$m \in \N$ there exists a unique $k \in H$ with $k^m = h$.
A $\Q$-powered hull of a $\T$-group $G$ is a $\Q$-powered group $H$
containing $G$ with the property that for every $h \in H$ there exists 
$m \in \N$ with $h^m \in G$. Malcev \cite{Mal51} observed that a 
$\Q$-powered hull of a $\T$-group exists and that it is unique up to isomorphism.
The construction of $\Q G$ implies the following; its elementary proof is 
left to the reader.

\begin{theorem}\label{Qpowered}
Let $G$ be a $\T$-group. Then $\Q G$ is a $\Q$-powered hull of $G$.
\end{theorem}

The following theorem summarizes the relationship between properties of 
pairs of $\T$-groups and isomorphisms of their completions. Recall that 
groups $G$ and $H$ are {\em commensurable} if there are subgroups of 
finite index $\ti{G} \leq G$ and $\ti{H} \leq H$ with $\ti{G} \cong \ti{H}$. 
The Malcev-correspondence associates a rational Lie algebra $\Lambda(G)$
with a $\T$-group $G$. This can be constructed explicitly via the 
inverse Baker-Campbell-Hausdorff formula, see \cite[Sec.\@ 2]{Pic71}
or \cite{ALi07} for details.

\begin{theorem}\label{main}
Let $G$ and $H$ be $\T$-groups. 
\begin{items}
\item[\rm (a)]
$F(G) = F(H)$ if and only if
$\Z_p G \cong \Z_p H$ for all primes $p$.
\item[\rm (b)]
$G$ and $H$ are commensurable if and only if $\Q G \cong \Q H$ or, 
equivalently, $\Lambda(G) \cong \Lambda(H)$.
\item[\rm (c)]
If $G$ and $H$ are commensurable, then 
$\Z_p G \cong \Z_p H$ for almost all primes $p$.
\item[\rm (d)]
If $\Z_p G \cong \Z_p H$ for a prime $p$, then $G$ and $H$ have 
the same type.
\end{items}
\end{theorem}

\begin{proof}
Pickel \cite{Pic71, Pic71a} proved similar theorems, but using a 
different definition for $\Z_p G$: For a finitely 
generated nilpotent group $G$, Pickel defines $\Z_p G$ as the completion 
of $G$ with respect to the uniform topology whose neighborhood basis of
the identity consists of the groups $\{ G^{p^i} \mid i \in \N \}$.
This completion coincides with the inverse limit of the quotients
$\{ G/G^{p^i} \mid i \in \N \}$.
If $G$ is a $\T$-group, then Pickel's definition for $\Z_p G$ coincides
with our definition, see Lemma 1.3 in \cite{Pic71}. Hence Pickel's results
apply in our setting. \\
(a) This is Lemma 1.2 in \cite{Pic71} and it is attributed to Borel. \\
(b)
This is often attributed to Malcev \cite{Mal51}. Alternatively, see 
\cite[p.\@ 330 \& Sec.\@ 2]{Pic71}. \\
(c)
Suppose that $G$ and $H$ are commensurable and let $\ti{G} \leq G$ and
$\ti{H} \leq H$ both of finite index in $G$ and $H$, respectively. If
$p$ is a prime with $p \nmid [G:\ti{G}]$ and $p \nmid [H:\ti{H}]$, then
Lemma 1.8 in \cite{Pic71} implies that $\Z_p G = \Z_p \ti{G} \cong 
\Z_p \ti{H} = \Z_p H$. \\
(d)
Suppose that $\varphi \colon \Z_p G \ra \Z_p H$ is an isomorphism. Then 
$\varphi(\Z_p I_i(G)) = \Z_p I_i(H)$, since $\Z_p \gamma_i(G) = 
\gamma_i( \Z_p G)$ is invariant under isomorphisms and further 
$\Z_p I_i(G) / \Z_p \gamma_i(G) \cong \Z_p ( I_i(G) / \gamma_i(G) )$
is the torsion subgroup of $\Z_p ( G / \gamma_i(G))$. Hence 
$\Z_p (I_i(G) / I_{i+1}(G)) \cong \Z_p (I_i(H) / I_{i+1}(H) )$
for each $i \in \N$, and $G$ and $H$ have the same type.
\end{proof}

\subsection{Isomorphisms}

The following result is the foundation to our approach towards testing 
for isomorphic finite quotients. Eick \& Engel \cite[Lemma 7]{EEn17} 
used a similar result to test isomorphisms of $\T$-groups.

\begin{theorem}\label{isom}
Let $t, s \in \Z(n)$ so that $G(t)$ and $G(s)$ are $\T$-groups of the
same type $(d_1,\ldots,d_c)$ with bases $(g_1, \ldots, g_n)$ and 
$(h_1, \ldots, h_n)$, respectively.
\begin{items}
\item[\rm (a)]
Let $p$ be a prime and $\varphi \colon \Z_p G(t) \ra \Z_p G(s)$ an 
isomorphism. Then $\varphi( \Z_p I_i(G(t)) = \Z_p I_i(G(s))$ for each 
$i \in \N$. 
\item[\rm (b)]
Let $p$ be a prime and $\varphi \colon \Z_p G(t) \ra \Z_p G(s)$ an 
isomorphism. Then for $1 \leq j \leq n$ it follows that $\varphi(g_j) 
= h^{m_j}$ for some $m_j = (m_{11}, \ldots, m_{1n}) \in \Z_p^n$ and the
matrix $M = (m_{ij})_{ 1 \le i,j \le n}$ is a block upper triangular
matrix with diagonal blocks $M_1, \ldots, M_c$ satisfying
$M_i \in \GL(d_i, \Z_p)$.
\item[\rm (c)] 
$F(G(t)) = F(G(s))$ if and only if for all primes $p$ there are
$m_1, \ldots, m_n \in \Z_p^n$ so that the elements $h^{m_1}, \ldots, 
h^{m_n} \in G(s)$ satisfy the relations of $G(t)$ and the matrix 
$M = (m_{ij})_{1 \le i,j \le n}$ is invertible over $\Z_p$.
\end{items}
\end{theorem}

\begin{proof}
(a) This is part of the proof of Theorem \ref{main} (d). \\
(b) Let $\varphi \colon \Z_p G(t) \ra \Z_p G(s)$ be an isomorphism. If $g_j 
\in I_i(G(t))$, then $\varphi(g_j) \in \Z_p I_i(G(s))$ by (a). Hence $M$ is 
in block upper triangular form and the blocks on the diagonal correspond 
to the isomorphisms $\varphi_i \colon \Z_p I_i(G(t)) / \Z_p I_{i+1}(G(t)) 
\ra \Z_p I_i(G(s)) / \Z_p I_{i+1}(G(s))$ induced by $\varphi$. Since
$\Z_p I_i(G(t)) / \Z_p I_{i+1}(G(t)) \cong \Z_p^{d_i} \cong 
\Z_p I_i(G(s)) / \Z_p I_{i+1}(G(s))$, each isomorphism $\varphi_i$ 
corresponds to an invertible $d_i \times d_i$-matrix over $\Z_p$. \\
(c) 
By (b) and Theorem \ref{main} (a) it suffices to prove the following for 
an arbitrary prime $p$: If there are $m_1, \ldots, m_n \in \Z_p^n$ so that 
the elements $h^{m_1}, \ldots, h^{m_n}$ satisfy the relations of $G(t)$ 
and the matrix of exponents $M = (m_{ij})_{1 \le i,j \le n}$ is invertible 
over $\Z_p$, then $\Z_p G(t) \cong \Z_p G(s)$ follows. 

Fix a prime $p$ and assume that $m_1, \ldots, m_n \in \Z_p^n$ exist. 
Let $\varphi \colon \{ g_1, \ldots, g_n \} \to \Z_p G(s), 
\ g_i \mapsto h^{m_i}$. As the images $\varphi(g_i)$ satisfy the relations 
of $G(t)$, it follows that $\varphi$ extends to an epimorphism 
$\tilde{\varphi} \colon G(t) \ra H \leq \Z_p G(s)$, where $H$ is the subgroup
generated by $h_1^{m_1}, \ldots, h_n^{m_n}$. As $M$ is invertible, it 
induces isomorphisms $I_i(G(t)) / I_{i+1}(G(t))$ onto $I_i(H)/I_{i+1}(H)$.
This implies that $\tilde{\varphi}$ is injective. Lemma \ref{lem:GroupRG}(b) 
now yields that $\tilde{\varphi}$ extends to an isomorphism $\Z_p G(t) \to 
\Z_p H$. Finally, $M$ is invertible, hence $\Z_p H = \Z_p G(s)$.
\end{proof}

Our general approach to classify $\T$-groups up to isomorphic finite quotients
is the following: Fix a type and consider indeterminates for $s$ and $t$.
Theorem \ref{isom}(c) translates the condition $F(G(s)) = F(G(t))$ into
a system of equations in $s,t$ and the entries of a block upper triangular 
matrix $M$.
It now remains to determine for which integral values of $s$ and $t$ this
system of equations is solvable over $\Z_p$ for all primes $p$.

\begin{lemma}
\label{twostep}
Let $G$ be a $\T$-group and let $C_i(G) = C_G(I_i(G)/I_{i+2}(G))$ be a 
two-step centralizer. Then $C_i(G)$ is characteristic in $G$ and 
$\Z_p C_i(G)$ is characteristic in $\Z_p G$.
\end{lemma}

\begin{proof}
We note that $\Z_p C_i(G) = C_{\Z_p G}(\Z_p I_i(G)/ \Z_p I_{i+2}(G))$. Now
the result follows, since centralizers of characteristic subgroups are
characteristic, see also Theorem \ref{isom}(a).
\end{proof}

\section{Proofs}\label{proofs}

This section contains the proofs for our results in Section \ref{results}.
We use the parameterized Hall polynomials for $\T$-groups of Hirsch length 
at most $5$, see \cite{CEi19}, the solution to the isomorphism problem for
these groups as described in \cite{EEn17} and the computer algebra system 
GAP \cite{GAP} as starting points. Our GAP code is publically available at
\cite{CEi23}.
Further, for a prime $p$ let $\nu_p \colon \Z_p \to \N_0 \cup \{ \infty \}$ 
denote the $p$-adic valuation.

\subsection{Proof of Theorem \ref{unbounded} and Corollary \ref{coro}}
\label{nobound}

Let $G$ be a $\T$-group of type $(2,1,1)$ so that the genus $\g(G)$
contains $m$ pairwise non-isomorphic groups $H_1, \ldots, H_m$ of type
$(2,1,1)$, see Section \ref{sec:Type211}, and let $A$ be free abelian of rank $\ell$. 
We show that $H_1 \times A, \ldots, H_m \times A$ are pairwise non-isomorphic 
groups in $\g(G \times A)$.

Consider $H \in \g(G)$. Then $\Z_p(H \times A) = \Z_pH \times \Z_pA$ by
construction for all primes $p$. This implies that $F(H \times A) = 
F(G \times A)$ by Theorem \ref{main}. Hence $H \times A \in 
\g(G \times A)$. 

The result that $H_i \times A \not \cong H_j \times A$ for $i \neq j$ 
follows from Hirshon \cite{Hir77}, since $H_i$ and $H_j$ both are of
type $(2,1,1)$ and thus have cyclic centers.
\hfill \fillbox

\subsection{$\T$-groups of type $(2,1,1)$}

\textit{Proof of Theorem \ref{class211}(b):}
Let $s, t \in T(2,1,1)$ and write $(g_1, \ldots, g_4)$ for a basis of
$G(t)$ and $(h_1, \ldots, h_4)$ for a basis of $G(s)$. Then $C_2(G(t))$ satisfies 
$G(t) > C_2(G(t)) > I_2(G(t))$ with $G(t)/C_2(G(t))$ infinite cyclic and 
$C_2(G(t))/I_2(G(t))$ infinite cyclic, see \cite[Lemma 8]{EEn17}. Using Theorem \ref{isom}
and Lemma \ref{twostep}, we assume that an isomorphism $\Z_p G(t) \ra \Z_p G(s)$ has
the form
\begin{alignat*}{3}
    g_1 & \ms & \ && h_1^{m_{11}} h_2^{m_{12}} h_3^{m_{13}} h_4^{m_{14}} \\
    g_2 & \ms & \ && h_2^{m_{22}} h_3^{m_{23}} h_4^{m_{24}} \\
    g_3 & \ms & \ && h_3^{m_{33}} h_4^{m_{34}} \\
    g_4 & \ms & \ && h_4^{m_{44}} 
\end{alignat*}
where $m_{ii}$ are invertible over $\Z_p$ and $m_{ij} \in \Z_p$. 
Evaluating the relations of $G(t)$ in $G(s)$ translates to the following
equations. (This translation can best be done by computer, see \cite{CEi23} 
for GAP code for this purpose.)
\begin{alignat*}{2}
    \text{(1)} & \quad & t_{123} m_{11}^{-1} m_{22}^{-1} m_{33} - s_{123} = & \ 0 \\
    \text{(2)} & \quad & t_{134} m_{11}^{-1} m_{33}^{-1} m_{44} - s_{134} = & \ 0 \\
    \text{(3)} & \quad & t_{124} m_{11}^{-1} m_{22}^{-1} m_{44} - s_{124} = & \
        s_{123} ( s_{134} (m_{11} - 1)/2 - m_{33}^{-1} m_{34} ) + s_{134} m_{22}^{-1} m_{23}           
\end{alignat*}
Equations (1--2) have solutions over $\Z_p$ for all primes $p$ if and 
only if $s_{123} = t_{123}$ and $s_{134} = t_{134}$, since $m_{ii} \in 
\Z_p^*$. In this case $m_{33} = m_{11} m_{22}$ and $m_{44} = m_{11}^2 m_{22}$ 
follow and it remains to consider Equation (3). This has a solution over
$\Z_p$ if and only if there are $x,y \in \Z_p$ and $u \in \Z_p^*$ with
\[ t_{124} u - s_{124} = s_{123} x + s_{134} y. \]
We note that if $x$ and $y$ range over $\Z_p$, then $s_{123} x + s_{134} y$ ranges
precisely over the $\Z_p$-multiples of $d(s) = \gcd(s_{123}, s_{134})$.
Let $k = \nu_p(d(s))$.
Then the $\Z_p$-multiples of $d(s)$ coincide with the $\Z_p$-multiples of $p^k$,
since $d(s)/p^k$ is a unit in $\Z_p$.

Hence a solution of Equation (3) implies the existence of a $w_p \in \Z$ with
$p \nmid w_p$ so that
\[ t_{124} w_p - s_{124} \equiv 0 \bmod p^k. \]
Vice versa, such an integer $w_p$ is a unit in $\Z_p$ and its existence shows that Equation (3) can be solved.
\hfill \fillbox

\subsection{$\T$-groups of type $(3,1,1)$}

\textit{Proof of Theorem \ref{class311}(b):}
Let $s, t \in T(3,1,1)$ and write $(g_1, \ldots, g_5)$ for the basis of
$G(t)$ and $(h_1, \ldots, h_5)$ for the basis of $G(s)$.
The presentation of $G(t)$ implies $C_1(G(t)) = \langle g_3, \ldots, g_5 \rangle$ and
$C_2(G(t)) = \langle g_2, \ldots, g_5 \rangle$.
Theorem \ref{isom} and Lemma \ref{twostep} show that an isomorphism $\Z_p G(t) \ra \Z_p G(s)$ has
the form
\begin{alignat*}{3}
    g_1 & \ms & \ && h_1^{m_{11}} h_2^{m_{12}} h_3^{m_{13}} h_4^{m_{14}} h_5^{m_{15}} \\
    g_2 & \ms & \ && h_2^{m_{22}} h_3^{m_{23}} h_4^{m_{24}} h_5^{m_{25}} \\
    g_3 & \ms & \ && h_3^{m_{33}} h_4^{m_{34}} h_5^{m_{35}} \\
    g_4 & \ms & \ && h_4^{m_{44}} h_5^{m_{45}} \\
    g_5 & \ms & \ && h_5^{m_{55}}
\end{alignat*}
where $m_{ii}$ are invertible over $\Z_p$ and $m_{ij} \in \Z_p$. 
Evaluating the relations of $G(t)$ in $G(s)$ translates to the following
equations.
\begin{alignat*}{2}
    \text{(1)} & \quad & t_{124} m_{11}^{-1} m_{22}^{-1} m_{44} - s_{124} = & \ 0 \\
    \text{(2)} & \quad & t_{145} m_{11}^{-1} m_{44}^{-1} m_{55} - s_{145} = & \ 0 \\
    \text{(3)} & \quad & t_{235} m_{22}^{-1} m_{33}^{-1} m_{55} - s_{235} = & \ 0 \\
    \text{(4)} & \quad & t_{135} m_{11}^{-1} m_{33}^{-1} m_{55} - s_{135} = & \ s_{145} m_{33}^{-1} m_{34} + s_{235} m_{11}^{-1} m_{12} \\
    \text{(5)} & \quad & t_{125} m_{11}^{-1} m_{22}^{-1} m_{55} - s_{125} = & \ s_{124} ( s_{145} ( m_{11} - 1 ) / 2 - m_{44}^{-1} m_{45} ) + s_{135} m_{22}^{-1} m_{23} \\
    & \quad & & + s_{145} m_{22}^{-1} m_{24} + s_{235} ( m_{11}^{-1} m_{22}^{-1} m_{12} m_{23} - m_{11}^{-1} m_{13} )
\end{alignat*}
Equations (1--3) have solutions over $\Z_p$ for all primes $p$ if and 
only if $s_{ijk} = t_{ijk}$ for $(ijk) \in \{ (124), (145), (235) \}$.
In this case $m_{44} = m_{11} m_{22}$ and $m_{55} = m_{11}^2 m_{22}$ 
follow.
If $s_{235} > 0$, then $m_{33} = m_{11}^2$ follows additionally.

It remains to consider Equations (4--5).
Note that $m_{11}^{-1} m_{33}^{-1} m_{55} = m_{11} m_{22} m_{33}^{-1}$
and $m_{11}^{-1} m_{22}^{-1} m_{55} = m_{11}$ can take arbitrary values in $\Z_p^*$ independently.
Further note that the right hand sides of Equations (4--5) independently range over all
$\Z_p$-multiples of $d_1(s) = \gcd(s_{145}, s_{235})$ and $d_2(s) = \gcd(s_{124}, s_{135}, d_1(s))$,
respectively.

Fix a prime $p$ and let $k_1 = \nu_p(d_1(s))$ and $k_2 = \nu_p(d_2(s))$.
Similar to our proof of Theorem \ref{class211} (b), we deduce that Equations (4--5) have a solution if and only if
there are $w_p, v_p \in \Z \setminus p \Z$ so that $p^{k_1} \mid (t_{135} w_p - s_{135})$ and $p^{k_2} \mid (t_{125} v_p - s_{125})$ hold. \hfill \fillbox

\subsection{$\T$-groups of type $(2,1,1,1)$}\label{subsec:Proof2111}

\textit{Proof of Theorem \ref{class2111}(b):}
Let $s, t \in T(2,1,1,1)$ and write $(g_1, \ldots, g_5)$ for the basis of
$G(t)$ and $(h_1, \ldots, h_5)$ for the basis of $G(s)$. The presentation 
of $G(t)$ implies $C_3(G(t)) = \langle g_2, \ldots, g_5 \rangle$. Thus an 
isomorphism $\Z_p G(t) \ra \Z_p G(s)$ has the same general form as for type 
$(3,1,1)$, see also Theorem \ref{isom} and Lemma \ref{twostep}. Evaluating 
the relations of $G(t)$ in the images of an isomorphism translates to 7
equations. We define
\begin{align*}
 v & = s_{123} ( m_{11} - 1) + m_{22}^{-1} m_{23} \\
 w & = s_{134} ( m_{11} - 1 ) / 2 + m_{33}^{-1} m_{34} \\
 x & = - m_{44}^{-1} m_{45} \\
 y & = m_{11}^{-1} m_{12} \\
 z_1 & = s_{134} s_{145} (m_{11}-1) (m_{11}-2) / 6 
         + s_{134} m_{44}^{-1} m_{45} ( m_{11} - 1) / 2 
         - s_{135} ( m_{11} - 1 ) / 2 \\
   & \ \ \ + s_{235} m_{12} + s_{235} ( m_{22} - 1 ) / 2  
         + m_{33}^{-1} ( m_{44}^{-1} m_{34} m_{45} - m_{35} ) \\
 z_2 & = s_{124} ( m_{11} - 1 ) / 2 
   + s_{134} m_{22}^{-1} m_{23} ( m_{11} - 1 ) / 2 + m_{22}^{-1} m_{24} \\
 z_3 & = m_{11}^{-1} ( m_{22}^{-1} m_{12} m_{23} - m_{13} )
\end{align*}
and, based on this, obtain the equations
\begin{alignat*}{2}
 \text{(1)} 
  & \quad & t_{123} m_{11}^{-1} m_{22}^{-1} m_{33} - s_{123} = & \ 0 \\
 \text{(2)} 
  & \quad & t_{134} m_{11}^{-1} m_{33}^{-1} m_{44} - s_{134} = & \ 0 \\
 \text{(3)} 
  & \quad & t_{145} m_{11}^{-1} m_{44}^{-1} m_{55} - s_{145} = & \ 0 \\
 \text{(4)} 
  & \quad & t_{235} m_{22}^{-1} m_{33}^{-1} m_{55} - s_{235} = & \ 0 \\
 \text{(5)} 
  & \quad & t_{124} m_{11}^{-1} m_{22}^{-1} m_{44} - s_{124} = & \ 
    s_{134} v - s_{123} w \\
 \text{(6)}  
  & \quad & t_{135} m_{11}^{-1} m_{33}^{-1} m_{55} - s_{135} = & \ 
    s_{134} x + s_{145} w + s_{235} y \\
 \text{(7)} 
  & \quad & t_{125} m_{11}^{-1} m_{22}^{-1} m_{55} - s_{125} = & \
    s_{124} x + s_{134} v x + s_{135} v + s_{123} z_1 + s_{145} z_2 
    + s_{235} z_3
\end{alignat*}
Equations (1--4) have solutions over $\Z_p$ for all primes $p$ if 
and only if $s_{ijk} = t_{ijk}$ holds for $(ijk) \in \{ (123), (134), 
(145), (235) \}$. In this case $m_{33} = m_{11} m_{22}$, $m_{44} = 
m_{11}^2 m_{22}$ and $m_{55} = m_{11}^3 m_{22}$ follow. If $s_{235} > 0$, 
then $m_{22} = m_{11}^2$ follows additionally.

It remains to consider Equations (5--7). Write $u = m_{11}$ to shorten
notation and note that $u = m_{11}^{-1} m_{22}^{-1} m_{44} 
= m_{11}^{-1} m_{33}^{-1} m_{55}$ and $u^2 = m_{11}^{-1} m_{22}^{-1} m_{55}$.
Next, note that if all $m_{ij}$ range independently over $\Z_p$, then so do 
$v$, $w$, $x$, $y$, $z_1$, $z_2$ and $z_3$. Further, $s_{123} z_1 + 
s_{145} z_2 + s_{235} z_3$ ranges over all $\Z_p$-multiples of $d_3(s) 
= \gcd(s_{123},s_{145},s_{235})$.

Fix a prime $p$. Then Equations (5--7) have a solution over $\Z_p$
if and only if there are $u \in \Z_p^*$ and $v, w, x, y, z \in \Z_p$ with
\begin{alignat*}{2}
 \rm{(i)} & \quad & s_{134} v = & \ t_{124} u - s_{124} + s_{123} w, \\
 \rm{(ii)} & \quad & s_{134} x = & \ t_{135} u - s_{135} - s_{145} w - s_{235} y,
         \quad \text{and} \\
 \rm{(iii)} & \quad & d_3(s) z = & \ t_{125} u^2 - s_{125} - s_{124} x 
       - s_{134} v x - s_{135} v.
\end{alignat*}

If Equations (i--iii) have a solution, then they can be solved modulo 
any power of $p$ by $u, v, w, x, y, z \in \Z$ with $p \nmid u$.
Vice versa, assume that $u, v, w, x, y, z \in \Z$ with $p \nmid u$ solve
Equations (i--iii) modulo $p^k = p^{\alpha + \beta}$,
where $\alpha = \nu_p(d_3(s))$ and $\beta = \nu_p(s_{134})$.
Then $u \in \Z_p^*$ and there are $A, B, C, D \in \Z$ with
\begin{align*}
   s_{134} v + A p^{\alpha + \beta} & = t_{124} u - s_{124} + s_{123} w, \\
   s_{134} x + B p^{\alpha + \beta} & = t_{135} u - s_{135} - s_{145} w - s_{235} y \quad \text{and} \\
   D p^\alpha = d_3(s) z + C p^{\alpha + \beta} & = t_{125} u^2 - s_{125} - s_{124} x - s_{134} v x - s_{135} v.
\end{align*}
It follows that $\tilde{u} = u$, $\tilde{w} = w$, $\tilde{y} = y$,
$\tilde{v} = v + A p^\alpha (s_{134}/p^\beta)^{-1}$ and
$\tilde{x} = x + B p^\alpha (s_{134}/p^\beta)^{-1}$ solve Equations 
(i--ii). If we insert these values in Equation (iii), then 
the right hand side translates to
\[ \Big( D - \big( s_{124} B + s_{134} (vB + xA + AB p^\alpha (s_{134}/p^\beta)^{-1}) + s_{135} A \big) (s_{134}/p^\beta)^{-1} \Big) p^\alpha =: E p^\alpha, \]
so $\tilde{z} = E (d_3(s)/p^\alpha)^{-1}$ yields a solution to Equation (iii).
\hfill \fillbox

\subsection{$\T$-groups of type $(2,1,2)$}\label{subsec:Proof212}

\textit{Proof of Theorem \ref{class212}(b):}
Let $s, t \in T(2,1,2)$ and $k = s_{235} / s_{134} \in \Z$.
Write $(g_1, \ldots, g_5)$ for the basis of $G(t)$ and $(h_1, \ldots, h_5)$ for the basis of $G(s)$.
Theorem \ref{isom} and Lemma \ref{twostep} show that an isomorphism $\Z_p G(t) \ra \Z_p G(s)$ has
the form
\begin{alignat*}{3}
    g_1 & \ms & \ && h_1^{m_{11}} h_2^{m_{12}} h_3^{m_{13}} h_4^{m_{14}} h_5^{m_{15}} \\
    g_2 & \ms & \ && h_1^{m_{21}} h_2^{m_{22}} h_3^{m_{23}} h_4^{m_{24}} h_5^{m_{25}} \\
    g_3 & \ms & \ && h_3^{m_{33}} h_4^{m_{34}} h_5^{m_{35}} \\
    g_4 & \ms & \ && h_4^{m_{44}} h_5^{m_{45}} \\
    g_5 & \ms & \ && h_4^{m_{54}} h_5^{m_{55}}
\end{alignat*}
where $m_{ij} \in \Z_p$ with $m_{33} \in \Z_p^*$ and
\[ M_1 = \begin{pmatrix} m_{11} & m_{12} \\ m_{21} & m_{22} \end{pmatrix}, \
    M_3 = \begin{pmatrix} m_{44} & m_{45} \\ m_{54} & m_{55} \end{pmatrix}
    \in \GL(2,\Z_p). \]
Evaluating the relations of $G(t)$ in $G(s)$ translates to the following (matrix)
equations.
\begin{alignat*}{2}
    \text{(1)} & \quad & t_{123} m_{33} = & \ s_{123} \det(M_1) \\
    \text{(2)} & \quad & \operatorname{diag}( t_{134}, t_{235} ) M_3 = & \ m_{33} M_1 \operatorname{diag}( s_{134}, s_{235} ) \\
    \text{(3)} & \quad & (t_{124}, t_{125}) M_3 - (s_{124}, s_{125}) \det(M_1) = & \ (m_{23}, -m_{13}) M_1 \operatorname{diag}(s_{134},s_{235}) + s_{123} (x,y)
\end{alignat*}
In Equation (3) we used the following definitions.
\begin{align*}
    x & = s_{134} ( m_{11}^2 m_{22} - m_{12} m_{21}^2 - \det( M_1 ) ) / 2 - \det(M_1) m_{33}^{-1} m_{34} \\
    y & = s_{235} ( m_{11} m_{22}^2 - m_{12}^2 m_{21} - \det( M_1 ) + 2 m_{12} m_{22} ( m_{11} - m_{21} ) ) / 2 - \det( M_1 ) m_{33}^{-1} m_{35}
\end{align*}

Equation (1) implies that $s_{123} = t_{123}$ and $m_{33} = \det(M_1)$.
If Equation (2) has a solution over $\Z_p$ for a prime $p$, then the 
diagonal matrices
$\operatorname{diag}( s_{134}, s_{235} )$ and 
$\operatorname{diag}( t_{134}, t_{235} )$
have the same elementary divisors over $\Z_p$. Thus
$\nu_p(s_{134}) = \nu_p(t_{134})$ and $\nu_p(s_{235}) = \nu_p(t_{235})$;
see also \cite[Sec.\@ IV, §7]{LomQui}.
As a consequence, Equation (2) has a solution for all primes $p$ if and 
only if $s_{134} = t_{134}$ and $s_{235} = t_{235}$. In this case, 
Equation (2) implies $m_{45} = k m_{12} m_{33}$ and thus $M_3 \in D_k(\Z_p)$.
We now fix a prime $p$ for the remainder of the proof.

If all $m_{ij}$ range independently over $\Z_p$, then so do $x$ and $y$.
Further $(m_{23}, -m_{13}) M_1$ ranges over all of $\Z_p^2$, since we 
assume that $M_1$ is invertible. It follows that the right hand side of 
Equation (3) takes arbitrary values of the form
$(a \gcd(s_{123},s_{134}), b \gcd(s_{123}, s_{235}))$ with $a,b \in \Z_p$,
and thus $(s_{124}, s_{125})$ and $(t_{124}, t_{125})$ are in the same 
orbit under the action of $D_k(\Z_p)$ on $L(s_{134}, s_{235}, s_{123})$.

Vice versa, assume that
$(t_{124}, t_{125}) A \equiv (s_{124}, s_{125}) 
\bmod (\gcd(s_{123},s_{134}), \gcd(s_{123}, s_{235}))$
for some $A \in D_k(\Z_p)$. It suffices to determine $M_1, M_3 \in 
\GL(2,\Z_p)$ with $\det(M_1)^{-1} M_3 = A$ that solve Equation (2).
Note that if $M_1$ and $M_3$ solve Equation (2), then $\det(M_3) = 
\det(M_1)^3$ holds, and thus $\det(M_1) = \det(A)$ and $M_3 = \det(A) A$.
This yields a solution to Equations (1--3) over $\Z_p$.
\hfill \fillbox

\begin{remark}\label{OrbitsOnL}
    Fix integers $a,b,c > 0$ with $a \mid b$ and a prime $p$, and let $k = b/a$.
    Let us give an explicit description of the orbits of $D_k(\Z_p)$ on $L(a,b,c)$.
    Let $(d,e), (f,g) \in L(a,b,c)$ and write $\ell = \nu_p( \gcd(a,c) )$ and $m = \nu_p( \gcd(b,c) )$.
    
    If there is a matrix $A \in D_k(\Z_p)$ with $(d,e) A \equiv (f,g) \mod (\gcd(a,c),\gcd(b,c))$,
    then there are $\alpha, \beta, \gamma, \delta \in \Z$ with
    \[ p^\ell \mid ( d \alpha + e \gamma - f ), \quad p^m \mid ( d k \beta + e \delta - g )
        \quad \text{and} \quad p \nmid (\alpha \delta - k \beta \gamma). \]
    
    The other way round, such integers $\alpha, \beta, \gamma, \delta$
    immediately yield a matrix
    \[ A = \begin{pmatrix} \alpha & k \beta \\ \gamma & \delta \end{pmatrix} \in D_k(\Z_p) \]
    with $(d,e) A \equiv (f,g) \mod (\gcd(a,c),\gcd(b,c))$.
\end{remark}

\bibliographystyle{abbrv}

\end{document}